\documentclass[12pt]{amsart}
\usepackage{amssymb,verbatim,enumerate,ifthen}
\usepackage[mathscr]{eucal}
\usepackage[utf8]{inputenc}
\usepackage[T1]{fontenc}
\usepackage{esint}
\usepackage{marginnote}
\usepackage{amsmath, amsthm}
\usepackage{lipsum}

\newtheorem{alphthm}{Theorem}

\newtheorem{thm}{Theorem}[section]

\newtheorem{prop}[thm]{Proposition}
\newtheorem{lem}[thm]{Lemma}

\theoremstyle{definition}

\newtheorem{exa}{Example}

\DeclareMathOperator{\conv}{conv}

\numberwithin{equation}{section}
\def\eq#1{{\rm(\ref{#1})}}
\def\Eq#1#2{\ifthenelse{\equal{#1}{*}}
  {\begin{equation*}\begin{aligned}[]#2\end{aligned}\end{equation*}}
  {\begin{equation}\begin{aligned}[]\label{#1}#2\end{aligned}\end{equation}}}

\def\QA#1#2{\mathscr{A}_{#1}^{[#2]}}

\newcommand{\floor}[1]{\left\lfloor #1 \right\rfloor}

\def\calI{\mathcal{I}}

\newcommand\R{\mathbb{R}}
\newcommand\N{\mathbb{N}}

\newcommand\calF{\mathcal{F}}

\DeclareMathOperator{\interior}{int}

\title[Equality problem for generalized quasiarithmetic means...]{Equality problem for generalized quasiarithmetic means generated by discontinuous strictly monotonic functions}

\author[Tibor Kiss]{Tibor Kiss}
\address{%
Institute of Mathematics, 
University of Debrecen,
Pf.\ 400, 4002 Debrecen, Hungary}
\email{kiss.tibor@science.unideb.hu}
\author{Pawe{\l} Pasteczka}
\address{Institute of Mathematics, University of Rzesz\'ow, Pigonia 1, 35-310 Rzesz\'ow, Poland}
\email{ppasteczka@ur.edu.pl}

\begin{document}

\begin{abstract}
We study the equality problem of generalized quasiarithmetic means for a strictly monotonic generator $f$ that is not necessarily continuous. We provide two sufficient conditions that lead to a conclusion analogous to the result of Páles and Pasteczka. We show through an example that in our case, without any extra conditions, the generator functions cannot be expected to be affine transformations of each other over the whole domain. In the remaining case, we consider the appropriate inverses of the functions, which implies that the scaling factor must coincide across the various regions of continuity.
\end{abstract}

\maketitle

\section{Introduction}
A function $M:\bigcup_{n=1}^\infty I^n\to I$ is called a \emph{mean on $I$} (here and below  $I\subseteq\R$ stands for a subinterval of positive length) if, for any $n\in\N$ and $x\in I^n$, the chain of inequalities
\Eq{*}{\min x\leq M(x)\leq \max(x)}
is satisfied. Among means, an actively researched family is the class of \emph{quasiarithmetic means} -- namely, functions $M$ for which there exists a strictly monotonic, continuous function $f:I\to\R$ such that
\Eq{*}{M(x)=f^{-1}\Big(\frac{f(x_1)+\dots+f(x_n)}n\Big)=:A_f(x),
\qquad n\in\N \text{ and }x=(x_1,\dots,x_n)\in I^n.}

The function $f$ is referred to as the generator of the mean. For a given $n\in \N$, we shall denote the restriction $A_f|_{I^n}$ by $A_f^{[n]}$. The framework for this family has been thoroughly treated in the literature, most notably in \cite{HarLitPol34} (as well as \cite{Bul03}, \cite{BulMitVas88}, and \cite{MitPecFin93}). 

Now we formulate a fundamental result concerning the equality of quasiarithmetic means.

\begin{alphthm}\label{thm:eq_qam}
Let $f,g:I\to\R$ be continuous, strictly monotonic functions. Then the following statements are pairwise equivalent.
\begin{enumerate}[(i)]
\item There exists $n\in\N\setminus\{1\}$ such that, for all $x\in I^n$, we have $A_f^{[n]}(x)=A_g^{[n]}(x)$.
\item For all $n\in\N$ and $x\in I^n$, we have $A_f^{[n]}(x)=A_g^{[n]}(x)$.
\item There exist constants $\alpha,\beta\in\R$ with $\alpha\neq 0$, such that $g=\alpha f+\beta$.
\end{enumerate}
\end{alphthm}

In a 2020 paper \cite{GruPal20}, Grünwald and Páles introduced a generalized left inverse $f^{(-1)}:\mathrm{conv}(f(I))\to I$ for a strictly monotonic function $f$ that is not necessarily continuous as the unique monotonic extension of the standard inverse (see Lemma~\ref{lem:GruPal} below). It allows to extend the family of quasiarithmetic means. Further results concerning this generalized inverse can also be found in \cite{PalPas26}.

We say that a mean $M:\bigcup_{n=1}^{\infty} I^n\to I$ is a \emph{generalized quasiarithmetic mean} if there exists a strictly monotonic function $f:I\to\R$ such that
\Eq{*}{
M(x)=f^{(-1)}\Big(\frac{f(x_1)+\dots+f(x_n)}n\Big)=:\mathscr{A}_f(x),
\qquad n\in\N \text{ and }x=(x_1,\dots,x_n)\in I^n.}

Similarly to the continuous case, the function $f$ is also referred to here as the generator of the generalized quasiarithmetic mean. For given $n\in \N$, the restriction of the function $\mathscr{A}_f$ to the set $I^n$ is denoted by $\QA{f}n$.

Clearly, if $f$ is continuous, then the generalized left inverse is simply the regular inverse, and hence the generalized quasiarithmetic mean is a regular one, that is $\mathscr{A}_f=A_f$. Furthermore, we have $\mathscr{A}_f=\mathscr{A}_{-f}$, and hence we may restrict our consideration to strictly increasing functions whenever convenient. 

Motivated by Theorem \ref{thm:eq_qam}, in their paper \cite{PalPas26}, Páles and Pasteczka addressed the equality problem of these generalized means. They arrived at an analogous result to that of the regular (continuous) case, but surprisingly, only under the assumption that the equality of the means in question holds for any number of variables. For the sake of exactness, we recall their result as well.

\begin{alphthm}\label{thm:eq_gqam}
Let $f,g:I\to\R$ be strictly monotonic functions. Then the following statements are pairwise equivalent.
\begin{enumerate}[(i)]
\item For all $n\in\N$ and $x\in I^n$, we have $\QA{f}n(x)=\QA{g}n(x)$.
\item There exist constants $\alpha,\beta\in\R$ with $\alpha\neq 0$, such that $g=\alpha f+\beta$.
\end{enumerate}
\end{alphthm}

The purpose of this paper is to determine a similar connection between the generator functions under the condition that equality is required only for a fixed number of variables. In our main theorem, we provide two sufficient conditions that lead to a conclusion analogous to the result of Páles and Pasteczka.

In what follows, we will study the problem when, for given strictly monotonic functions $f,g \colon I \to \R$ and fixed $n \in \N$, the equality $\QA{f}n(x_1,\dots,x_n)=\QA{g}n(x_1,\dots,x_n)$
is satisfied for all $x_1,\dots,x_n\in  I$. For the sake of brevity, we will refer to this simply as $\QA{f}n=\QA{g}n$. 

We show through an example that in our case, without any extra conditions, the generator functions cannot be expected to be affine transformations of each other. More precisely, we provide an example of two generators $f$ and $g$ that are not affine transformations of each other, yet there exists an $n\geq 2$ such that $\QA{f}n=\QA{g}n$ is valid; see Example~\ref{exa:2436} below.

Remarkably, due to \cite[Corollary 4.1]{PalPas26} we know that whenever $\QA{f}n=\QA{g}n$ holds for some $n \in \N$ then $\QA{f}m=\QA{g}m$ is valid for all $m \le n$. Taking into account the trivial inequality $\QA{f}1=\QA{g}1$, for every pair of strictly increasing functions $f,g \colon I\to \R$, we can define the maximal $n$ such that $\QA{f}n=\QA{g}n$. In view of Theorem~\ref{thm:eq_gqam}, $n=+\infty$ corresponds to the case when $f$ and $g$ are affine transformations of each other.

Therefore, it is reasonable to study whether, for a given pair $(f,g)$ of strictly increasing functions, there exists a finite number $n\in \N$ such that $\QA fn=\QA gn$ implies that $f$ and $g$ are affine transformations of each other. If so, to give an upper bound for its value.

The aim of this paper is to deliver a thorough study of this equality problem.

Let us underline that generalized quasiarithmetic means are a subclass of semideviation means  (cf. \cite{Pal89b}, \cite{PalPas19a}). 
Analogous problems within the family of Gini means \cite{Gin38} were studied by Losonczi \cite{Los70a}. For the Bajrakarevi\'c means \cite{Baj58}, this problem was studied in Acz\'el--Dar\'oczy \cite{AczDar63c}. Let us mention that the case of the equality of bivariate means and equality for the means for any number of variables are very different. The equality problem for a fixed number of variables greater than two remains open.

For a locally bounded function $f:I\to\R$, define $f_-,f_+:I\to\R$ by
\Eq{*}{
  f_-(x):=\lim_{r\to0+}\inf_{u\in(x-r,x+r)\cap I}f(u)
  \quad\mbox{and}\quad
  f_+(x):=\lim_{r\to0+}\sup_{u\in(x-r,x+r)\cap I}f(u)\qquad(x\in I).
}
Then, one can easily see that $f_-$ is the largest lower semicontinuous function below $f$ and $f_+$ is the smallest upper semicontinuous function above $f$. 

If $f:I\to\R$ is a strictly increasing continuous function, then its image $f(I)$ is itself an open interval, and $f$ is a homeomorphism between $I$ and $f(I)$. In particular, the inverse function $f^{-1}$ is continuous, strictly increasing, and maps $f(I)$ onto $I$. In this situation we also have $f_-=f_+=f$.

In the more general setting where $f:I\to\R$ is merely strictly increasing (without assuming continuity), the mapping $f:I\to f(I)$ is still bijective, and $f$ is locally bounded. Because $I$ is an open interval, $f$ admits both a left and a right limit at each point of $I$, and we see that
\Eq{*}{
  f_-(x)=\lim_{u\to x-} f(u),
  \qquad\text{ and }\qquad
  f_+(x)=\lim_{u\to x+} f(u)
  \qquad (x\in I).
}
It is straightforward to verify that $f_-$ and $f_+$ are strictly increasing functions and, for any $u<x<v$ in $I$, we have
\Eq{*}{
   f_+(u)<f_-(x)\le f(x)\le f_+(x)<f_-(v).
}
We denote by $C_f$ the set of points of $I$ at which $f$ is continuous. It is a classical fact that $C_f$ is co-countable in $I$, that is, $I\setminus C_f$ is a countable set. It is also evident that $x\in C_f$ holds if and only if $f_-(x)=f(x)=f_+(x)$. Analogous statements hold for strictly decreasing functions.

For a subset $S\subseteq\R$, we denote by $\conv(S)$ the smallest convex set containing $S$, which in this context coincides with the smallest interval containing $S$. For our definition of generalized quasiarithmetic means, we will use the following lemma concerning the existence and properties of the generalized inverse of strictly increasing (not necessarily continuous) functions.

\begin{lem}[\cite{GruPal20}, Lemma 1]\label{lem:GruPal} Let $f:I\to\R$ be strictly increasing. Then there exists a uniquely determined increasing function $f^{(-1)}:\conv(f(I))\to I$ that serves as a left inverse of $f$, i.e., for every $x\in I$,
\Eq{SMF}{
   (f^{(-1)}\circ f)(x)=x.
}
Moreover, the following statements hold:
\begin{enumerate}[(i)]
 \item The function $f^{(-1)}$ is continuous;
 \item For every $y\in f(I)$,
 \Eq{*}{
  (f\circ f^{(-1)})(y)=y;
 }
 \item For each $y\in\conv(f(I))$,
 \Eq{*}{
  (f_-\circ f^{(-1)})(y)\leq y \leq (f_+\circ f^{(-1)})(y);
 }
 \item For all $x\in I$,
 \Eq{*}{
   (f^{(-1)}\circ f_-)(x)=(f^{(-1)}\circ f_+)(x)=x,
 }
 which is equivalent to
 \Eq{*}{
   (f_-)^{(-1)}=(f_+)^{(-1)}=f^{(-1)}.
 }
\end{enumerate}
\end{lem}

Observe that assertion (iv) is not stated explicitly in \cite[Lemma 1]{GruPal20}, but it follows directly from \eq{SMF} by taking left and right limits at $x$ and applying assertion (i).

The function $f^{(-1)}$ introduced in the lemma above is called the \emph{generalized inverse of the strictly increasing function $f:I\to\R$}. Assertions (ii) and (iii) show that the restriction of $f^{(-1)}$ to $f(I)$ coincides with the ordinary inverse of $f$. Consequently, $f^{(-1)}$ can be viewed as the unique continuous and increasing extension of the usual inverse of $f$ to the minimal interval containing the range of $f$.

In what follows, we recall a simple lemma from \cite{PalPas26} concerning the properties of the generalized inverse. In what follows, we will use these properties without a direct reference to this statement. 
\begin{lem}
Let $f:I\to\R$ be a strictly increasing function, $x\in I$, and $u\in\conv(f(I))$. Then we have the following equivalences:
\begin{enumerate}[(i)]
 \item $f^{(-1)}(u)=x$ holds if and only if $u\in[f_-(x),f_+(x)]$.
 \item $f^{(-1)}(u)<x$ holds if and only if $u<f_-(x)$.
 \item $f^{(-1)}(u)\leq x$ holds if and only if $u\leq f_+(x)$.
 \item $x<f^{(-1)}(u)$ holds if and only if $f_+(x)<u$.
 \item $x\leq f^{(-1)}(u)$ holds if and only if $f_-(x)\leq u$.
\end{enumerate}
\end{lem}

\section{Properties of means}
Let us begin this section with a classical property of means. We say that $M \colon I^n \to I$ is \emph{internal} (or \emph{strict}) if for every nonconstant vector $x \in I^n$ we have $\min(x)<M(x)<\max(x)$. Internality plays an important role in invariant means and iterations of mean-type mappings (cf. for example Borwein--Borwein \cite{BorBor87} and Jarczyk--Jarczyk \cite{JarJar18} and references therein).

At this point, we introduce several nonstandard definitions (properties) for symmetric means. We will use them only in the restricted version when the indicated mean is a generalized quasiarithmetic one; however, there is no need to restrict these definitions to this family. 

Due to \cite{PalPas26}, we know that whenever the generalized quasiarithmetic mean is internal, then its generator is continuous and therefore it is a standard quasiarithmetic mean. Furthermore, all quasiarithmetic means are internal. Now we define the related property: an internality point. 
Let $M\colon I^n \to I$ be a symmetric mean. We say that $x \in I$ is an \emph{internality point} of $M$ if for all $y_1,\dots,y_{n-1} \in I$ we have
$M(x,y_1,\dots,y_{n-1})=x$ implies either $y_1=\dots=y_{n-1}=x$ or $\min(y_1,\dots,y_{n-1})<x<\max(y_1,\dots,y_{n-1})$. Note that if a mean is strict, then every point of $I$ is an internality point. In what follows, for a given mean $M$, let $\calI(M)$ be the set of all its internality points.

Now we show that for generalized quasiarithmetic means internality points are precisely these points where the generator is continuous. This result corresponds to \cite[Lemma~2.5]{PalPas26}.
\begin{prop}
Let $f \colon I \to \R$ be strictly increasing. Then for all $n \in \N$ with $n \ge 2$ we have  $\calI(\QA{f}n)=C_f$.
\end{prop}
\begin{proof}
Assume that $x\in C_f$ and take $y_1,\dots,y_{n-1}\in I$ such that $\QA{f}n(x,y_1,\dots,y_{n-1})=x$. Then 
\Eq{*}{
f(x)+f(y_1)+\dots+f(y_{n-1})=nf(x),
}
and therefore
\Eq{*}{
f(x)=\frac{f(y_1)+\dots+f(y_{n-1})}n.
}
Thus, since the arithmetic mean is strict we obtain that $f(x)=f(y_1)=\dots=f(y_{n-1})$  or  $\min(f(y_1),\dots,f(y_{n-1})) < f(x) <\max(f(y_1),\dots,f(y_{n-1}))$. Since $f$ is strictly increasing, the same is true for the arguments, which shows that $x$ is an internality point. 

Assume now that $x\in I \setminus C_f$. Then $f_-(x)<f(x)$ or $f_+(x)>f(x)$. Furthermore, assume without loss of generality that the first inequality is valid. Then $f_-(x)>\frac1{n-1}(nf_-(x)-f(x))=:c$, and therefore there exists $y<x$ such that $f(y)>c$. Thus we have
\Eq{*}{
f_+(x)\ge f(x) \ge \frac{f(x)+(n-1)f(y)}n>\frac{f(x)+(n-1)c}n=f_-(x),
}
and therefore $\QA{f}n(x,y,\dots,y)=x$, i.e. $x \notin \calI(\QA{f}n)$.
\end{proof}

Now we introduce the notion of attractivity and passability. Let $M \colon I^n\to I$ be a symmetric mean. We say that $V\subset I$ is \emph{$(k,M)$-attracting} (or briefly \emph{$k$-attracting} if the indicated mean is known) if for all $x_1,\dots,x_k \in V$ there exists $x_{k+1},\dots,x_n \in I$ such that $M(x_1,\dots,x_n)\in V$. It is clear that all intervals meet this property, and hence it is important mostly when $V$ is a disconnected set.

We say that $t$ is a \emph{passage point} of $M$ if  there exist $y_1,\dots,y_{n-1},z_1,\dots,z_{n-1} \in I$ such that for all $i \in\{1,\dots,n\}$ either $y_i=z_i$ or ($y_i<z_i$ and $[y_i,z_i] \subset \calI(M)$), and
\Eq{*}{
M(t,y_1,\dots,y_{n-1})<t<M(t,z_1,\dots,z_{n-1}).
}
If this property is valid for all $t \in I$, then $M$ is called \emph{passable}.

Note that there exist generalized quasiarithmetic means that do not possess a passage point. If the generator is strictly increasing and its points of discontinuity are everywhere dense in the interval $I$, then the interior of the set of internality points of the mean is empty; therefore, it is naturally impossible to find a subinterval on which the mean is continuous. 

Furthermore, even generalized quasiarithmetic means with only a finite number of discontinuities could not be passable. To see this, consider the following elementary situation. Let $I:=(a,b)$ with $a<b$ and $f:I\to\mathbb{R}$ be a positive strictly increasing function with the only discontinuity point $t=\frac{a+b}{2}$. Assume further that $2f_-(t)<f(t)$. Then we had $y\in I$ with $\QA{f}2(t,y)<t$ if and only if we would have $f(y)=2f_-(t)-f(t)$, that is, $f(y)<0$. This contradicts the positivity of $f$. Finally note that every point belonging to the interior of $C_f$ is a passage point.

\section{Main results}

We start with a series of propositions which are related to the fact that when $\QA{f}n=\QA{g}n$ yields that  $f$ is an affine transformation of $g$. In the final theorem we will use all of them.

\begin{prop}\label{prop:1}
Let $f,g \colon I \to \R$ be strictly increasing. Assume that there exists $n\in\N \setminus\{1\}$ such that $\QA{f}n=\QA{g}n$
and $\interior \calI(\QA{f}n)$ is $(2,\QA{f}n)$-attracting. Then there exists $\alpha\in \R$ such that $f-\alpha g$ is constant on every connected component of $ \interior \calI(\QA{f}n)$.
\end{prop}
\begin{proof}
Define $C:=\calI(\QA{f}n)$ and
\Eq{*}{
\calF&:=\bigg\{(\xi_1,\dots,\xi_n)\in I^n \colon \xi_1,\xi_2\in \interior C, \frac{f(\xi_1)+\dots+f(\xi_n)}{n} \in f(\interior C)\bigg\}.
}
Then by $\QA{f}n=\QA{g}n$ we also have $C=\calI(\QA{g}n)$ and hence
\Eq{*}{
\bigg\{(\xi_1,\dots,\xi_n)&\in I^n \colon \xi_1,\xi_2\in \interior C,\frac{g(\xi_1)+\dots+g(\xi_n)}{n} \in g(\interior C)\bigg\}\\
&=\bigg\{(\xi_1,\dots,\xi_n)\in I^n \colon \xi_1,\xi_2\in \interior C,\  \QA{g}n(\xi_1,\dots,\xi_n) \in \interior C\bigg\}\\
&=\bigg\{(\xi_1,\dots,\xi_n)\in I^n \colon \xi_1,\xi_2\in \interior C,\  \QA{f}n(\xi_1,\dots,\xi_n) \in \interior C\bigg\}\\
&=\bigg\{(\xi_1,\dots,\xi_n)\in I^n \colon \xi_1,\xi_2\in \interior C, \frac{f(\xi_1)+\dots+f(\xi_n)}{n} \in f(\interior C)\bigg\}=\calF.
}

We show that for all $\xi,\eta \in \calF$ we have
\Eq{fxi=gxi}{
f(\xi_1)+\dots+f(\xi_n)&=f(\eta_1)+\dots+f(\eta_n) \\
&\iff g(\xi_1)+\dots+g(\xi_n)=g(\eta_1)+\dots+g(\eta_n).
}
As a matter of fact, we only show the $(\Rightarrow)$ implication as the second one is completely analogous. Assume that $\xi,\eta \in \calF$ and
denote briefly
\Eq{*}{
u:=\frac{f(\xi_1)+\dots+f(\xi_n)}n;&\qquad
v:=\frac{f(\eta_1)+\dots+f(\eta_n)}n;\\
u':=\frac{g(\xi_1)+\dots+g(\xi_n)}n;&\qquad
v':=\frac{g(\eta_1)+\dots+g(\eta_n)}n.
}
Then we have $u,v \in f(C)$ and therefore $f^{(-1)}(u)=f^{-1}(u)$ and $f^{(-1)}(v)=f^{-1}(v)$. 
By $\QA{f}n=\QA{g}n$, we have 
\Eq{*}{
f^{-1}(u)&=f^{(-1)}(u)=f^{(-1)}\bigg(\frac{f(\xi_1)+\dots+f(\xi_n)}n\bigg)
=\QA{f}n(\xi_1,\dots,\xi_n)=\QA{g}n(\xi_1,\dots,\xi_n)\\
&=g^{(-1)}\bigg(\frac{g(\xi_1)+\dots+g(\xi_n)}n\bigg)=g^{(-1)}(u').
}
Therefore, since $\xi \in \calF$ we obtain $u' \in g(C)$ and therefore $f^{-1}(u)=g^{-1}(u')$. Similarly we get $v' \in C$ and $f^{-1}(v)=g^{-1}(v')$. Thus we obtain that $u=v$ implies 
\Eq{*}{
g^{-1}(u')=f^{-1}(u)=f^{-1}(v)=g^{-1}(v')
}
and therefore $u'=v'$. In this way we have proved that \eq{fxi=gxi} is valid.

Now we proceed to the main part of the proof. First, since we need to study values of generating functions, we take an inverse mapping of elements in $\calF$. To this end, let us denote
\Eq{*}{
\Delta:=\{p \in f(I)^n \colon (f^{-1}(p_1),\dots,f^{-1}(p_n)) \in \calF\}.
}
Then, in view of \eq{fxi=gxi}, for every $p,q \in \Delta$ with $p_1+\dots+p_n=q_1+\dots+q_n$ we have
\Eq{*}{
g\circ f^{-1}(p_1)+\dots+g\circ f^{-1}(p_n)=g\circ f^{-1}(q_1)+\dots+g\circ f^{-1}(q_n).
}
Thus we can define the function $\Phi \colon \{p_1+\dots+p_n\colon p \in \Delta\} \to \R$ by 
\Eq{*}{
\Phi(x):=g\circ f^{-1}(p_1)+\dots+g\circ f^{-1}(p_n), \text{ where } x=p_1+\dots+p_n.
}

Then there exists an open neighbourhood $U_{p_1}$ of $p_1$ and $U_{p_2}$ of $p_2$ such that $(s_1,s_2,\dots,p_n)\in \Delta$ for all $s_1 \in U_{p_1}$ and $s_2 \in U_{p_2}$ and  we have
\Eq{*}{
\Phi(s_1+s_2+p_3+\dots+p_n)=g\circ f^{-1}(s_1)+g\circ f^{-1}(s_2)+g\circ f^{-1}(p_3)\dots+g\circ f^{-1}(p_n)
}
hence
\Eq{3.3}{
\Phi(s_1+s_2+p_3+\dots+p_n)&-\Phi(x)\\
&=g\circ f^{-1}(s_1)-g\circ f^{-1}(p_1)+g\circ f^{-1}(s_2)-g\circ f^{-1}(p_2).
}
Now we use the fact that, by the definition of $\Delta$, both $f$ and $g$ are continuous on both $U_{p_1}$ and $U_{p_2}$. Then, since $\QA fn=\QA gn$, using Theorem~\ref{thm:eq_qam}, there exist $\alpha_{p_1},\alpha_{p_2},\beta_{p_1},\beta_{p_2} \in \R$ such that 
\Eq{*}{
g\circ f^{-1}(s)=\alpha_{p_1} s+\beta_{p_1}\text{ for all }s \in U_{p_1}\\
g\circ f^{-1}(s)=\alpha_{p_2} s+\beta_{p_2}\text{ for all }s \in U_{p_2}.
}
Then \eq{3.3} yields that for all $s_1\in U_{p_1}$ and $s_2\in U_{p_2}$ (recall $(p_1,\dots,p_n)\in \Delta$) we have
\Eq{*}{
\Phi(s_1+s_2+p_3+\dots+p_n)-\Phi(x)=\alpha_{p_1}(s_1-p_1)+\alpha_{p_2}(s_2-p_2).
}
Hence, by considering $s_1 \to p_1$; $s_2=p_2$, and $s_1=p_1$; $s_2\to p_2$, we obtain $\alpha_{p_1}=\Phi'(x)=\alpha_{p_2}$. Consequently, the value of $\alpha_p$ is independent of $p$; say that they are all equal to $\alpha$. Thus for all $s_1 \in U_{p_1}$ and $s_2 \in U_{p_2}$  we have
\Eq{*}{
\Phi(s_1+s_2+p_3+\dots+p_n)-\Phi(x)=\alpha(s_1-p_1+s_2-p_2)
}

Since $\interior C$ is $(2,\QA{f}n)$-attracting, for all $\xi_1,\xi_2 \in \interior C$, there exist $\xi_3,\dots,\xi_n \in I$ such that $(\xi_1,\dots,\xi_n)\in \calF$. Therefore, we obtain that
$g(x)-\alpha f(x)$ is constant on every connected component of $\interior C$.
\end{proof}

In this example, we show that the assumptions of Proposition~\ref{prop:1} are insufficient to guarantee the generators are affine transformations of each other over the whole domain.
\begin{exa}\label{exa:2436}
Define the strictly increasing functions $f,g:(0,3)\to\R$ by
\Eq{*}{
f(t)=
\begin{cases}
t&\text{if }t\in(0,1],\\
t+2&\text{if }t\in(1,2],\\
t+4&\text{if }t\in(2,3)
\end{cases}
\qquad\text{and}\qquad
g(t)=
\begin{cases}
t&\text{if }t\in(0,1],\\
t+3&\text{if }t\in(1,2],\\
t+6&\text{if }t\in(2,3)
\end{cases}
}

We show that $\QA f3=\QA g3$. Take $x_1,x_2,x_3 \in (0,3)$ with $x_1\le x_2 \le x_3$ and let $A_1:=(0,1]$, $A_2:=(1,2]$ and $A_3:=(2,3)$. Then
\Eq{*}{
f(A_1)=(0,1];\qquad f(A_2)=(3,4];\qquad f(A_3)=(6,7);\\
g(A_1)=(0,1];\qquad g(A_2)=(4,5];\qquad g(A_3)=(8,9).
}

If $x_1,x_2,x_3 \in A_j$ for some $j \in\{1,2,3\}$ then we clearly have $\QA f3(x_1,x_2,x_3)=\frac{x_1+x_2+x_3}3=\QA g3(x_1,x_2,x_3)$.

If $x_1,x_2 \in A_1$ and $x_3 \in A_2 \cup A_3$ then $f(x_1)+f(x_2)+f(x_3)\in (3,9)$, and therefore $\frac{f(x_1)+f(x_2)+f(x_3)}3\in (1,3)=(f_-(1),f_+(1))$, that is $\QA f3(x_1,x_2,x_3)=1$. Similarly $g(x_1)+g(x_2)+g(x_3) \in(4,11)$ which yields that $\frac{g(x_1)+g(x_2)+g(x_3)}3\in (1,4)=(g_-(1),g_+(1))$. Thus we get $\QA g3(x_1,x_2,x_3)=1$.

Using the fact that $f(x+1)=f(x)+2$ and $g(x+1)=g(x)+3$ we obtain that for all $x_1,x_2,x_3 \in (0,2)$ we have $\QA{f}3(x_1+1,x_2+1,x_3+1)=\QA{f}3(x_1,x_2,x_3)+1$ and $\QA{g}3(x_1+1,x_2+1,x_3+1)=\QA{g}3(x_1,x_2,x_3)+1$. Thus, based on the above case, we obtain that $\QA f3(x_1,x_2,x_3)=\QA g3(x_1,x_2,x_3)$ for $x_1,x_2 \in A_2$ and $x_3 \in A_3$.

Now let $x_1 \in A_1 \cup A_2$ and $x_2,x_3 \in A_3$. Analogously to the previous case we get \Eq{*}{
f(x_1)+f(x_2)+f(x_3) &\in (12,18)= (3f_-(2),3f_+(2))\\
g(x_1)+g(x_2)+g(x_3) &\in (16,23)\subset (3g_-(2),3g_+(2))\\
}
and thus $\QA{f}3(x_1,x_2,x_3)=2=\QA{g}3(x_1,x_2,x_3)$.

For $x_1 \in A_1$ and $x_2,x_3 \in A_2$ we get 
\Eq{*}{
f(x_1)+f(x_2)+f(x_3) \in (6,9] \subset (3f_-(1),3f_+(1)]\\
g(x_1)+g(x_2)+g(x_3) \in (8,11] \subset (3g_-(1),3g_+(1)]
}
and  hence $\QA{f}3(x_1,x_2,x_3)=1=\QA{g}3(x_1,x_2,x_3)$.

Finally if $x_1 \in A_1$, $x_2 \in A_2$ and $x_3\in A_3$ then
\Eq{*}{
f(x_1)+f(x_2)+f(x_3) \in (9,12) = (3f_+(1),3f_-(2))\\
g(x_1)+g(x_2)+g(x_3) \in (12,15) = (3g_+(1),3g_-(2)).
}
Therefore, we have
\Eq{*}{
\QA{f}3(x_1,x_2,x_3)&=f^{-1}\Big(\frac{f(x_1)+f(x_2)+f(x_3)}3\Big)=\frac{x_1+x_2+2+x_3+4}3-2=\frac{x_1+x_2+x_3}3;\\
\QA{g}3(x_1,x_2,x_3)&=g^{-1}\Big(\frac{g(x_1)+g(x_2)+g(x_3)}3\Big)=\frac{x_1+x_2+3+x_3+6}3-3=\frac{x_1+x_2+x_3}3.
}

Thus, we have established that $\QA{f}3 = \QA{g}3$. In the remaining case, we consider the appropriate inverses of the functions, which implies that the scaling factor must coincide across the various regions of continuity.

Now we show that $\interior \calI(\QA{f}3)=(0,1)\cup (1,2)\cup(2,3)$ is
$(2,\QA{f}3)$-attracting. If two elements belong to different connected components that we take the third one from the remaining component and we are in the last case. Otherwise, we can take the third element to be equal to the one of the first two and we are done.

In this way we showed that the assumption of Proposition~\ref{prop:1} are satisfied with $n=3$ even thought $f-g$ is not constant on the entire domain.
\end{exa}

The next proposition provides the sufficient condition to keep the affine transformation in the discontinuity points. 

\begin{prop}\label{prop:2}
Let $f,g \colon I \to \R$ be strictly increasing. Assume that there exists $n\in\N \setminus\{1\}$ such that $\QA{f}n=\QA{g}n$ holds 
and $\interior \calI(\QA{f}n)$ is $(2,\QA{f}n)$-attracting.

Then for every passage point $t$ of $\QA{f}n$, there exists an open neighbourhood $V$ of $t$ and $\alpha \in \R$ such that $f-\alpha g$ is constant on  $\calI(\QA{f}n) \cap V$.
\end{prop}
\begin{proof}
Since $t$ is a passage point of $\QA fn$, there exist $y_1,\dots,y_{n-1},z_1,\dots,z_{n-1} \in I$ such that for all $i \in\{1,\dots,n\}$ either $y_i=z_i$ or ($y_i<z_i$ and $[y_i,z_i] \subset \calI(\QA fn)$), and
\Eq{*}{
\QA fn(t,y_1,\dots,y_{n-1})<t<\QA fn(t,z_1,\dots,z_{n-1}).
}

By Proposition~\ref{prop:1}, we know that there exists $\alpha \in \R$ such that $f-\alpha g$ is constant on each interval $A_i:=[y_i,z_i]$ (which could be also a singleton), say $f=\alpha g+\beta_i$ on $A_i$ for all $i \in \{1,\dots,n-1\}$. Then for every $(t_1,\dots,t_{n-1})\in A_1\times\dots\times A_{n-1}$ and $x \in I$ we have
\Eq{*}{
\QA{f}n(x,t_1,\dots,t_{n-1})&=f^{(-1)}\bigg(\frac{f(x)+f(t_1)+\dots+f(t_{n-1})}n\bigg)\\
&=f^{(-1)}\bigg(\frac{f(x)+\alpha(g(t_1)+\dots+g(t_{n-1}))+\beta_1+\dots+\beta_{n-1}}n\bigg)
}
Let $m \colon [0,1]\to I$ be defined as 
\Eq{*}{
m(\theta):=\QA{f}n(t,(1-\theta) y_1+\theta z_1,\dots,(1-\theta) y_{n-1}+\theta z_{n-1})
}
Then $m$ is nondecresing in both variables, and $m(0)<t<m(1)$. 
Therefore the mapping $k \colon [0,1] \to \R$ given by
\Eq{*}{
k(\theta):= f(t)+f((1-\theta) y_1+\theta z_1)+\dots+f((1-\theta) y_{n-1}+\theta z_{n-1})
}
is continuous. Moreover we have $m(\theta)=f^{(-1)}(\frac1n k(\theta))$, and therefore $m$ is continuous too.

Take $\rho \in (m(0),m(1))$ arbitrarily and define
\Eq{*}{
c_{1,\rho}&:=\inf\{\theta \in [0,1]\colon m(\theta)\ge \rho\}=\inf\{\theta \in [0,1]\colon m(\theta)= \rho\};\\
c_{2,\rho}&:=\sup\{\theta \in [0,1]\colon m(\theta)\le \rho\}=\sup\{\theta \in [0,1]\colon m(\theta)= \rho\}.
}
Then
\Eq{*}{
k(c_{1,\rho})=nf_-(\rho) \qquad \text{ and }\qquad k(c_{2,\rho})=nf_+(\rho).
}
Therefore
\Eq{*}{
nf_-(\rho)=f(t)+f((1-c_{1,\rho}) y_1+c_{1,\rho} z_1)+\dots+f((1-c_{1,\rho}) y_{n-1}+c_{1,\rho} z_{n-1})\\
nf_+(\rho)=f(t)+f((1-c_{2,\rho}) y_1+c_{2,\rho} z_1)+\dots+f((1-c_{2,\rho}) y_{n-1}+c_{2,\rho} z_{n-1})
}
Similarly we obtain
\Eq{*}{
ng_-(\rho)=g(t)+g((1-c_{1,\rho}) y_1+c_{1,\rho} z_1)+\dots+g((1-c_{1,\rho}) y_{n-1}+c_{1,\rho} z_{n-1})\\
ng_+(\rho)=g(t)+g((1-c_{2,\rho}) y_1+c_{2,\rho} z_1)+\dots+g((1-c_{2,\rho}) y_{n-1}+c_{2,\rho} z_{n-1})
}
Then we have
\Eq{226}{
n(f_-(\rho)-\alpha g_-(\rho))=(f(t)- \alpha g(t))+\beta_1+\dots+\beta_{n-1}\\
n(f_+(\rho)-\alpha g_+(\rho))=(f(t)- \alpha g(t))+\beta_1+\dots+\beta_{n-1}
}
Thus there exists $\beta \in \R$ such that
\Eq{*}{
\beta:=f_-(\rho)-\alpha g_-(\rho)=f_+(\rho)-\alpha g_+(\rho).
}

Therefore $f_--\alpha g_-$ and $f_+-\alpha g_+$ are constant and equal to $\beta$ on $(m(0),m(1))$. Therefore $f-\alpha g=\beta$ on $\calI(\QA{f}n)\cap V$, where $V=(m(0),m(1))$ is a neighbourhood of $t$.
\end{proof}

In the two propositions above, we have gained insight into the behavior of $f-\alpha g$ in the internal and passage points, respectively. In the final theorem, we combine them. Let us emphasize that this result is not just a straightforward corollary of Propositions~\ref{prop:1} and \ref{prop:2}.

\begin{thm}\label{thm:fAgB}
Let $f,g \colon I \to \R$ be strictly increasing. Assume that there exists $n\in\N \setminus\{1\}$ such that $\QA{f}n=\QA{g}n$ holds, $\QA{f}n$ is passable,  
and $\interior \calI(\QA{f}n)$ is $(2,\QA{f}n)$-attracting. 

Then there exist $\alpha,\beta \in \R$ such that $f(x)=\alpha g(x)+\beta$ for all $x \in I$.
\end{thm}
\begin{proof}
By Proposition~\ref{prop:2} there exists $\alpha$ such that for every $x \in I$ there exists an open neighbourhood $V_x$ of $x$ such that $f-\alpha g$ is constant on $\calI(\QA{f}n) \cap V_x$. 

Let $a,b \in \calI(\QA{f}n)$ and $\{V_x \colon x \in I\}$ be an open covering of $[a,b]$. Take a finite subcover $V_{x_1},\dots,V_{x_n}$ of $[a,b]$. Since $f-\alpha g$ is constant on each $\calI(\QA{f}n)\cap V_{x_i}$, we obtain $f(a)-\alpha g(a)=f(b)-\alpha g(b)$. 
Therefore there exists $\beta \in \R$ such that $f(x)=\alpha g(x)+\beta$ for all $x \in \calI(\QA{f}n)$.  

Since $\calI(\QA{f}n)$ is $(2,\QA{f}n)$-attracting, for all $x\in I$ and $t_1 \in \calI(\QA{f}n)$ there exist $t_2,\dots,t_{n-1} \in \calI(\QA{f}n)$ such that
\Eq{*}{
m:=\QA{f}n(x,t_1,\dots,t_{n-1}) \in \calI(\QA{f}n).
}
Then
\Eq{*}{
f(x)=nf(m)-(f(t_1)+\dots+f(t_{n-1})).
}
Analogously, by $\QA{f}n=\QA{g}n$ we get
\Eq{*}{
g(x)&=ng(m)-(g(t_1)+\dots+g(t_{n-1})).
}
Therefore
\Eq{*}{
f(x)&=n(\alpha g(m)+\beta)-((\alpha g(x_1)+\beta)+\dots+(\alpha g(x_{n-1})+\beta))\\
&=\alpha(ng(m)-(g(t_1)+\dots+g(t_{n-1})))+\beta=\alpha g(x)+\beta,
}
which concludes the proof.
\end{proof}

\section{Examples/applications}

Motivated by Theorem~\ref{thm:fAgB} we can ask if for every strictly increasing function $f \colon I \to \R$ there exists a (finite) number $n$ such that whenever $g \colon I \to \R$ is strictly increasing and $\QA{f}n=\QA{g}n$ then $g$ is an affine transformation of $f$. In the next example we show that this conjecture is false.
\begin{exa}\label{exa:1}
For every nondecreasing sequence $(a_n)_{n=0}^\infty$ with $a_0 = 2$ and a sequence $(b_n)_{n=0}^\infty$ with all coefficients in the interval $(0,1)$,  let $f_{a,b} \colon (0,\infty) \to \R$ be given by 
\Eq{*}{
f_{a,b}(t):=\floor{t}+b_{\floor{t}}(t-\floor{t})+(\floor{t}+1)!a_{\floor{t}}.
}
To shorten the notation, we will use the mantissa function $\{t\}:=t-\floor{t}$ and define the sequence $(u_a(n))_{n=0}^\infty$ by $u_a(n):= (n+1)!a_n$. Then, for all $t \in (0,\infty)$, we have $f_{a,b}(t)=\floor{t}+b_{\floor{t}}\{t\}+u_a(\floor{t})$. 

First we show that no pair of functions $(f_{a,b},f_{a',b'})$ are affine transformation of each other. Suppose that there exists $\alpha,\beta \in \R$ such that $f_{a,b}=\alpha f_{a',b'}+\beta$. For $k \in \N$ we have 
\Eq{*}{
f_{a,b}(k+\tfrac12)-f_{a,b}(k)&=\alpha(f_{a',b'}(k+\tfrac12)-f_{a',b'}(k)),\\
\frac 12 b_k&=\alpha \frac12 b'_k.
}
Thus $b_k=\alpha b'_k$ for all $k\in\N$.
Furthermore 
\Eq{*}{
f_{a,b}(k)-\lim_{t\to k^-}f_{a,b}(t)&=k+(k+1)!a_k-\big((k-1)+b_{k-1}+k!a_{k-1}\big)\\
&=1-b_{k-1}+(k+1)!a_k-k!a_{k-1},\\
f_{a',b'}(k)-\lim_{t\to k^-}f_{a',b'}(t)&
=1-b'_{k-1}+(k+1)!a'_k-k!a'_{k-1}.
}
Thus
\Eq{*}{
0&=f_{a,b}(k)-\lim_{t\to k^-}f_{a,b}(t)-\alpha \Big( f_{a',b'}(k)-\lim_{t\to k^-}f_{a',b'}(t)\Big)\\
&=1-\alpha -(b_{k-1}-\alpha b'_{k-1})
+(k+1)!(a_k-\alpha a_k')-k!(a_{k-1}-\alpha a'_{k-1})\\
&=1-\alpha +(k+1)!(a_k-\alpha a_k')-k!(a_{k-1}-\alpha a'_{k-1}).
}
Therefore
\Eq{*}{
(k+1)!(a_k-\alpha a_k')=\alpha-1+k!(a_{k-1}-\alpha a'_{k-1}).
}
But 
\Eq{*}{
(k+1)!(a_k-\alpha a_k')=f_{a,b}(k)-\alpha f_{a',b'}(k)=\beta.
}

Therefore $\alpha-1=0$, and hence $\alpha=1$.  Thus $b_k=b_k'$ and $\beta=f_{a,b}(t)-f_{a',b'}(t)$, i.e., this function is constant. But
\Eq{*}{
\beta=f_{a,b}(\tfrac12)-f_{a',b'}(\tfrac12)=\tfrac12 b_0+a_0-\tfrac12 b_0'-a_0'=0.
}
In this way we have proved that $f_{a,b}=f_{a',b'}$. Using this equality at the integer points, we get that $a=a'$.

Next observe that, for all $n \ge 2$,
\Eq{FGs}{
u_a(n)&=(n+1)! a_n = n n! a_{n}+n!\cdot a_n= nu_a(n-1)+n!\cdot a_n \\
&\ge nu_a(n-1)+n (n-1).
}
Thus
\Eq{*}{ 
u_a(n)-u_a(n-1) \ge 
(n-1)(u_a(n-1)+n) \ge 2n \text{ for all }n \ge 3.
}
Furthermore
\Eq{*}{
u_a(1)-u_a(0) &\ge 2a_1-a_0\ge a_1 \ge 2,\\
u_a(2)-u_a(1) &\ge 6a_2-2a_1=4a_2+2(a_2-a_1) \ge 4,
}
and therefore we get
\Eq{FGs2}{ 
u_a(n)-u_a(n-1) \ge 2n\text{ for all }n \ge 1.
}

Next, fix $n \in \N$ and let $x \in [0,\infty)^n$ be an arbitrary vector with $x_1\le \dots\le x_n$. If $x_n \ge n$ then we have
\Eq{*}{
\frac{f_{a,b}(x_1)+\dots+f_{a,b}(x_n)}n &\ge \frac{f_{a,b}(x_n)}n \ge \frac{x_n
+u_a(\floor{x_n})}n \ge \frac{\floor{x_n}
+u_a(\floor{x_n})}{\floor{x_n}} = 1+\frac{u_a(\floor{x_n})}{\floor{x_n}}.
}
Thus \eq{FGs} yields
\Eq{*}{
\frac{f_{a,b}(x_1)+\dots+f_{a,b}(x_n)}n  \ge \floor{x_n}+u_a(\floor{x_n}-1) =(f_a)_-(\floor{x_n}).
}
On the other hand if $\floor{x_1}<\floor{x_n}$ and $x_n\ge n$ then we have
\Eq{*}{
\frac{f_{a,b}(x_1)+\dots+f_{a,b}(x_n)}n &= \frac{x_1+b_{\floor{x_1}}\{x_1\}+u_a(\floor{x_1})+\dots+x_n +b_{\floor{x_n}}\{x_n\}+u_a(\floor{x_n})}n\\
&\le \frac{nx_n+n+u_a(\floor{x_1})+(n-1)u_a(\floor{x_n})}n\\
&\le \frac{n\floor{x_n}+2n+nu_a(\floor{x_n})+u_a(\floor{x_1})-u_a(\floor{x_n})}n\\
&\le \floor{x_n}+u_a(\floor{x_n})+ \frac{2n-(u_a(\floor{x_n})-u_a(\floor{x_n}-1))}n.
}
Therefore \eq{FGs2} implies
\Eq{*}{
\frac{f_a(x_1)+\dots+f_a(x_n)}n
&\le \floor{x_n}+u_a(\floor{x_n}) =f_{a,b}(\floor{x_n}).
}
Finally whenever $\max(x_1,\dots,x_n) \ge n$ we have 
\Eq{*}{
\QA{f_{a,b}}n(x_1,\dots,x_n)=\begin{cases}
\floor{\max(x_1,\dots,x_n)} & \text{ if }\floor{\min(x_1,\dots,x_n)}<\floor{\max(x_1,\dots,x_n)};\\
\frac{x_1+\dots+x_n}n & \text{ if }\floor{\min(x_1,\dots,x_n)}=\floor{\max(x_1,\dots,x_n)}.
\end{cases}
}
Therefore, in this case, the value of $\QA{f_{a,b}}n(x_1,\dots,x_n)$ does not depend on sequences $a$ and $b$. Furthermore, if $\max(x_1,\dots,x_n) < n$ then the value of $\QA{f_{a,b}}n(x_1,\dots,x_n)$ depends only on the first $n$ elements of sequences $a$ and $b$. 

Consequently, for every two nondecreasing sequences $a,a^*$ and every two sequences $b,b^*$  with all coefficients in the interval $(0,1)$  such that $a$ and $b$ coincide with $a^*$ and $b^*$ up to the index $n$, respectively, we have
 $\QA{f_{a,b}}n=\QA{f_{a^*,b^*}}n$.
\end{exa}

The next example is the extension of the one presented in the proof of \cite[Proposition~5.1]{PalPas26}. 
\begin{exa}
Let $a,b \in \R$ with $a<0<b$, $\alpha,\beta \in (0,\infty)$ and let $f \colon [a,b] \to \R$ be given by 
\Eq{*}{
f(x)=\begin{cases}
    -1+\alpha x &\text{for }x\in [a,0),\\
    0 &\text{for }x=0,\\
    1+\beta x &\text{for }x\in (0,b].\\
\end{cases}
}
We are searching for the minimal value of $n$ such that for every strictly increasing function $g \colon [a,b] \to \R$ the equality $\QA fn=\QA gn$ yields that $g$ is an affine transformation of $f$. We will use Theorem~\ref{thm:fAgB}. To this end, we need to find the value of $n$ such that   $\QA{f}n$ is passable,  
and $\interior \calI(\QA{f}n)$ is $(2,\QA{f}n)$-attracting.

To ensure attractivity it is sufficient that for all $m \in (-1,1)$ we have
\Eq{*}{
2m+(n-2)f(a)<nf_-(0) \text{ or }2m+(n-2)f(b)>nf_+(0),
}
equivalently
\Eq{*}{
2m<nf_-(0)-(n-2)f(a) \text{ or }2m>nf_+(0)-(n-2)f(b).
}
Therefore, it is sufficient that 
\Eq{*}{
nf_+(0)-(n-2)f(b)&<nf_-(0)-(n-2)f(a);\\
n-(n-2)(1+\beta b)&<-n-(n-2)(-1+\alpha a);\\
-\beta b n+2 (1+\beta b)&<-\alpha a n+2 (\alpha a-1);\\
(\alpha a-\beta b) n&<2(\alpha a -\beta b-2);\\
n&>2 \frac{2+\beta b-\alpha a}{\beta b-\alpha a}=2+\frac{4}{\beta b-\alpha a}.
}

To verify the passability, we shall show that $0$ is a passage point of $\QA{f}n$. Note that $\calI(\QA{f}n)=(a,0) \cup (0,b)$ is an open set. Thus, considering the endpoints of these intervals, it is sufficient to show that there exists $k \in \{1,\dots,n-1\}$ and $\varepsilon>0$ such that 
\Eq{*}{
\QA fn (0,\underbrace{a+\varepsilon,\dots,a+\varepsilon}_{k\text{ times}},\underbrace{\varepsilon,\dots,\varepsilon}_{n-k-1\text{ times}})<0<\QA fn (0,\underbrace{-\varepsilon,\dots,-\varepsilon}_{k\text{ times}},\underbrace{b-\varepsilon,\dots,b-\varepsilon}_{n-k-1\text{ times}}).
}
Equivalently, we get the following two inequalities
\Eq{*}{
kf(a+\varepsilon)+(n-k-1)f(\varepsilon)&<nf_-(0),\\
kf(-\varepsilon)+(n-k-1)f(b-\varepsilon)&>nf_+(0),
}
that is
\Eq{*}{
k(-1+\alpha(a+\varepsilon))+(n-k-1)(1+\beta\varepsilon)&<-n,\\
k(-1+\alpha(-\varepsilon))+(n-k-1)(1+\beta(b-\varepsilon))&>n.
}
Since both sides of the latter inequalities are continuous with respect to $\varepsilon$ and the inequalities are strict, it is sufficient to verify if they are valid for $\varepsilon=0$ (also with the strict inequality signs). Thus we verify the condition
\Eq{*}{
k(-1+\alpha a)+n-k-1&<-n;\\
-k+(n-k-1)(1+\beta b)&>n.
}
After obvious rearranging, we get
\Eq{*}{
(-2+\alpha a)k&<-2n+1;\\
(-2-\beta b)k&>-\beta b n+(1+\beta b).
}
Using the fact that $\alpha a<0$ and $\beta b>0$ we get $-2+\alpha a <0$ and $-2-\beta b<0$, and thus we can rewrite these inequalities in the equivalent forms
\Eq{*}{
k>\frac{2n-1}{2-\alpha a}\qquad\text{ and }\qquad k<\frac{\beta b n-(1+\beta b)}{2+\beta b}.
}
Now we must verify that there exists $k \in\{1,\dots,n\}$ which satisfies these inequalities. As a matter of fact, we are unable to provide a complete characterization; therefore, we focus on a sufficient condition. It is sufficient that the following three inequalities are valid
\Eq{*}{
\frac{2n-1}{2-\alpha a}&<n,\\
\frac{\beta b n-(1+\beta b)}{2+\beta b} &>1,\\
\frac{2n-1}{2-\alpha a}+1&<\frac{\beta b n-(1+\beta b)}{2+\beta b}.
}
Since $\alpha a<0$ we know that the first inequality is trivially valid.
The remaining two conditions are equivalent to
\Eq{*}{
n>\frac{3}{\beta b}+2,\quad \text{ and }\quad
2(4+\alpha \beta ab)n
< (2a\alpha-3)(2b\beta+3)+1,
}
respectively. 
Under the assumption $\alpha a \beta b <-4$, the third assertion is equivalent to
\Eq{*}{
n>\frac{3\alpha a+2\alpha a\beta b-3\beta b-4}{4+\alpha a \beta b}.
}
Thus, under the assumption $\alpha a \beta b <-4$, we get that the condition
\Eq{*}{
n > \max\bigg\{2+\frac{4}{\beta b-\alpha a},\frac{3}{\beta b}+2, 
\frac{3\alpha a+2\alpha a\beta b-3\beta b-4}{4+\alpha a \beta b}
\bigg\}
}
is sufficient to provide that the equality $\QA{f}n=\QA{g}n$ implies that $g$ is an affine combination of $f$. Even though the inequality $\alpha a \beta b <-4$ is not necessarily valid, this example is a~typical application of Theorem~\ref{thm:fAgB}.
\end{exa}

\section*{Discussion}

\begin{enumerate}[1.]
    \item Our results do not cover the case when $I\setminus C_f$ is dense. This problem remains open.
    \item Let us remind that, in view of Example~\ref{exa:1} it is not true that for every pair $(f,g)$ there exists a finite $n\in \N$ such that $\QA fn=\QA gn$ yields that $f$ is an affine transformation of $g$. Therefore each statement of this type would require some extra assumptions. 
    \item The notion of internality points appeared naturally to determine the continuity points of the generating function; however, it is reasonable to study it for a different family of means. Furthermore, it could have application in generalizing results concerning invariance of mean-type mappings. 
    \item It is quite clear that the only interesting case for using our result is when both $f$ and $g$ are not continuous; however, it is not assumed explicitly. 
    \item Due to \cite{PalPas26}, we know that the equality of weighted means implies that the generators are the affine transformations of each other, and therefore none of our results has its weighted counterpart. 
    \item The problem of giving the necessary and sufficient condition for $\QA fn=\QA gn$ in terms of generating functions remains open. 
    \item In this paper, we put some computational effort into the examples. However, in our opinion, they are necessary to understand the background of the indicated problem. 
    \item Last, but not least, let us mention that this family possesses several natural superclasses (for example, generalized Bajraktarevi\'c means). This paper also contributes to the discussion of the equality problem among them.  
\end{enumerate}

\subsection*{Funding} The research of the first author was supported by the Tempus Public Foundation and HUN-REN Hungarian Research Network.
\subsection*{Data Availability Statement} Not applicable.
\subsection*{Competing interests} The authors have no relevant financial or non-financial interests to disclose.

\end{document}